\documentclass[12pt]{amsart}
\makeatletter
\@namedef{subjclassname@2020}{\textup{2020} Mathematics Subject Classification}
\makeatother
\usepackage{amssymb}
\usepackage{graphicx}
\usepackage{tikz}
\usepackage{pgfplots} \pgfplotsset{compat=newest}

\theoremstyle{plain}
\newtheorem{theorem}{Theorem}[section]
\newtheorem{MainTheorem}{Theorem}

\newtheorem{lemma}[theorem]{Lemma}
\newtheorem{proposition}[theorem]{Proposition}
\newtheorem{conjecture}[theorem]{Conjecture}
\theoremstyle{definition}

\renewcommand{\rho}{\varrho}
\renewcommand{\kappa}{\varkappa}
\newcommand{\eps}{\varepsilon}
\renewcommand{\phi}{\varphi}
\newcommand{\R}{\mathbb{R}}

\newcommand{\var}{\operatorname{Var}}
\newcommand{\rohat}{\widehat{\rho}}

\newcommand{\rotil}{\widetilde{\rho}}
\newcommand{\pov}{\overline{p}}
\newcommand{\roov}{\overline{\rho}}

\newcommand{\x}{\mathfrak{X}}
\newcommand{\y}{\mathfrak{Y}}

\newcommand{\abs}[1]{\left|{#1}\right|}

\begin{document}
\title{Flexibility of entropies for piecewise expanding unimodal maps}
\author{Llu\'{\i}s Alsed\`a}
\address{Departament de Matem\`atiques and Centre de Recerca Matem\`atica,
Edifici Cc,
Universitat Aut\`onoma de Barcelona,
08913 Cerdanyola del Vall\`es,
Barcelona, Spain}
\email{alseda@mat.uab.cat}

\author{Micha{\l} Misiurewicz}
\address{Department of Mathematical Sciences, IUPUI, 402 N. Blackford
  Street, Indianapolis, IN 46202, USA}
\email{mmisiure@math.iupui.edu}

\author{Rodrigo A. P\'erez}
\address{Department of Mathematical Sciences, IUPUI. 402 N. Blackford
  Street, Indianapolis, IN 46202, USA}
\email{rperez@math.iupui.edu}

\thanks{This work was partially supported by grant number 426602
from the Simons Foundation to Micha{\l} Misiurewicz and by the
Spanish grant MTM2017-86795-C3-1-P from the
``Agencia Estatal de Investigaci\'on'' (AEI).
Micha{\l} Misiurewicz also acknowledges the support and hospitality of the
Centre de Recerca Matem\`atica, Barcelona, Spain.
}

\date{February 26, 2020}



\begin{abstract}
We investigate the \emph{flexibility} of the entropy
(topological and metric) for the class of piecewise
expanding unimodal maps.
We show that the only restrictions for the values of
the topological and metric entropies in this class
are that both are positive,
the topological entropy is at most $\log 2$,
and by the Variational Principle,
the metric entropy is not larger than the topological entropy.

In order to have a better control on the metric entropy, we work
mainly with topologically mixing piecewise expanding skew tent maps,
for which there are only 2 different slopes.
For those maps, there is an additional
restriction that the topological entropy is larger than
$\frac{1}{2}\log2$.

We also make the interesting observation that for skew tent maps
the sum of reciprocals of derivatives of all iterates of the map at the
critical value is zero. It is a generalization and a different
interpretation of the Milnor-Thurston formula connecting the
topological entropy and the kneading determinant for unimodal maps.
\end{abstract}

\maketitle

\section{Introduction}\label{s-intro}

Recently an important program in Dynamical Systems was initiated by
Anatole Katok. It concerns \emph{flexibility}, that is, the idea that
for a given class of dynamical systems, dynamical invariants (for
instance entropies) can take arbitrary values, subject only to natural
restrictions. Various results in this direction were obtained for
instance in papers~\cite{E, EK, BKRH}.

Here we investigate the family of piecewise expanding unimodal maps.
While they are not smooth, they are piecewise smooth (in fact, the
maps that we consider are piecewise linear). For those maps,
by~\cite{LaY}, there exists an absolutely continuous invariant
probability measure. By~\cite{LiY}, this measure is unique. Therefore
we can consider its metric entropy (which is also equal to its
Lyapunov exponent), as well as the topological entropy of the map.
Both entropies are positive, topological entropy is at most $\log 2$,
and by the Variational Principle, the metric entropy is not larger
than the topological entropy. We will show (Theorem~\ref{flex-uni})
that those are the only restrictions for the values of those
entropies.

In order to have a better control on the metric entropy, we will work
mainly with piecewise expanding skew tent maps, for which there are
only 2 different slopes. In particular, for the topologically mixing
expanding skew tent maps we prove a version of a theorem on
flexibility of entropies. For those maps, there is the additional
restriction that the topological entropy is larger than
$\frac12\log2$. Again, it turns out that there are no additional
restrictions (Theorem~\ref{flex-skew}).

There are two basic things that we have to prove in order to get
Theorem~\ref{flex-skew}. One is continuity of the density of the
absolutely continuous invariant probability measure as a function of a
map, and the other one is existence of maps with small metric entropy.
For this, we need strong estimates on the density of this measure.
Classical methods, initiated in~\cite{LaY}, using the variation and
integral, are difficult and give too weak estimates. In particular,
in~\cite{BK} continuity of the density as a function of the map is
proved only at maps for which the turning point is not periodic.

The problem with this classical approach is that it is difficult to
trace the trajectory of a density under the iterates of the
Frobenius-Perron operator. This is due to the fact that most points
have two preimages and the value of the image of the density at a
given point depends on the values (and derivatives, that are usually
different) at those preimages. However, for skew tent maps there is an
alternative to this procedure. Instead of looking at the whole
density, we look only at its jumps (discontinuities). Those jumps
propagate along one trajectory of the map, and it is easy to keep
track of them.

As the old saying goes, \emph{nihil sub sole novum},\footnote{Eccles.
  1:10 (Vulg.)} and this method has been employed by Ito, Tanaka and
Nakada~\cite{ITN} over 40 years ago. They obtained a simple formula
for the densities of absolutely continuous invariant measures for skew
tent maps. Their results are not as widely known as they deserve,
probably due to the fact that the term ``skew tent map'' was not used
at that time.\footnote{A search in the MathSciNet suggests that it was
  used for the first time in~\cite{MV}.}

Besides proving our main theorems about flexibility and theorems that
lead to them, we make an interesting observation. Namely, the formula
of Milnor and Thurston~\cite{MT}, connecting for unimodal maps the
kneading sequence to the topological entropy, can be reinterpreted
easily as the fact that for a tent map the sum of reciprocals of
derivatives of all iterates of the map at the critical value is zero.
It turns out that this is also true for skew tent maps. We wonder
whether this can be translated back into a language involving some
entropy-like quantities.

The paper is organized as follows. In Section~\ref{s-def} we give the
basic definitions. In Section~\ref{s-dens1} we prove that if the
positive slope of a mixing expanding skew tent map is large then the
density of the absolutely continuous invariant probability measure is
close to 1. In Section~\ref{s-dens} we prove continuous dependence of
this density on the map, while in Section~\ref{s-entr} we prove
continuous dependence of the metric entropy on the map. We do it for a
larger class of maps, namely, we do not assume mixing. In
Section~\ref{s-root} we modify the standard square root construction
(basically inverse of the renormalization process) in order to stay in
the class of piecewise expanding maps. Then, in Section~\ref{s-main}
we prove our main theorems, and in Section~\ref{s-obs} we make the
observation we mentioned.

\section{Definitions}\label{s-def}

An interval map $f:[0,1]\to[0,1]$ is called \emph{piecewise expanding}
if there is a finite partition of $[0,1]$ into smaller intervals, and
on the closure of each of those smaller intervals $f$ is of class
$C^2$ and $|f'|\ge T$ for some constant $T>1$. If $f$ is unimodal,
this partition can be finer than the partition into pieces of
monotonicity (laps).

For a unimodal map $f:[0,1]\to[0,1]$ we assume that $f$ is increasing
on the left lap and decreasing on the right one. If $c$ is the turning
(critical) point for $f$ then the \emph{core} of $f$ is the interval
$[f^2(c),f(c)]$.

A \emph{skew tent} map is a unimodal map which is linear (we will use
this term in the sense of ``affine''; this is a common terminology
in real analysis) on each lap (see Figure~\ref{skew-tent}).
There are three popular models for skew
tent maps. In the first one the map $f$ is defined on some interval
containing 0 in its interior, 0 is the turning point, and $f(0)=1$
(\cite{MV}). In
the second and third ones $f$ maps $[0,1]$ to itself. In the second
model, $f(0)=f(1)=0$ (\cite{BK}). In the third model, $[0,1]$ is the
core of $f$ (\cite{ITN}).
We will use the third model. Thus, in particular, we will have
$f(c)=1$ and $f(1)=0$. The slopes of $f$ will be denoted by $s$ (the
left slope) and $-t$ (the right slope). The condition that $f$ maps
$[0,1]$ to itself translates to the condition $\frac1s+\frac1t\ge 1$.

\begin{figure}
\begin{tikzpicture}
\draw (0,0) -- (0,5) -- (5,5) -- (5,0) -- (0,0) -- (5,5);
\draw[blue, very thick] (5,0) -- (1,5) -- (0,1);
\draw[dashed] (1.02,-0.1) -- (1.02,5);
\draw[dashed] (2.7777777,-0.1) -- (2.7777777, 2.8);
\node[below] at (0, 0) {\tiny $0$}; \node[left] at (0, 0) {\tiny $0$};
\node[below] at (5, 0) {\tiny $1$}; \node[left] at (0, 5) {\tiny $1$};
\node[below] at (1, 0) {\tiny $c=\tfrac{t-1}{t}$};
\node[below] at (2.7777777, 0) {\tiny $\tfrac{t}{t+1}$};
\node[right, rotate=-51.4] at (1.6,4.7) {\footnotesize slope $-t$};
\node[left,rotate=76] at (0.6,4.5) {\footnotesize slope $s$};
\end{tikzpicture}
\caption{A skew tent map.}\label{skew-tent}
\end{figure}
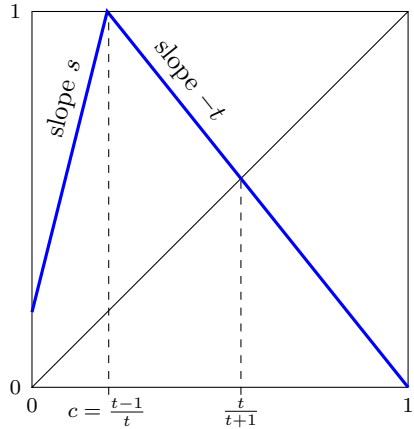

The two main spaces of skew tent maps that we consider are the space $\y$
of all piecewise expanding skew tent maps and its subspace $\x$
consisting of topologically mixing expanding skew tent maps.
Topological mixing is equivalent to each of two other conditions: that
the topological entropy is larger than $\frac12\log2$, and that $f(0)$
is to the left of the fixed point of $f$ (there is a unique fixed
point, except when $f(0)=0$, and then we mean the other fixed point).
In terms of slopes, each map of $\y$ is determined by slopes $s,t>1$,
subject to $\frac1s+\frac1t\ge 1$, while each map of $\x$ is
determined by slopes $s,t>1$, restricted to $t>\frac1s+\frac1t\ge 1$.
We consider the spaces $\y$ and $\x$ with the uniform (sup) topology.
Parametrization via the absolute values of the slopes gives the same
topology.

For computations it is good to remember that $s,t>1$, that $c=(t-1)/t$,
and that the fixed point is $t/(t+1)$ (see again Figure~\ref{skew-tent}).

By~\cite{LaY}, piecewise expanding maps have an absolutely continuous
invariant probability measure (\emph{acip} in short). By~\cite{LiY},
if the map is unimodal, this measure is unique. We will use the notation
$\mu$ for this measure and $\rho$ for its density. For positive
multiples of this measure (that is, absolutely continuous invariant
measures, that are finite, but not necessarily normed) we will use the
acronym \emph{acim}. Of course, by ``absolutely continuous'' we mean
absolutely continuous with respect to the Lebesgue measure $\lambda$.

We assume that the reader knows some basics (we really mean only
basics) of the kneading theory. If not, a quick look into~\cite{CE}
or~\cite{MT} can be useful.

\section{Density is close to $1$}\label{s-dens1}

The following theorem is proved in~\cite{ITN}
(see Figure~\ref{Densities} for some examples of
normalized densities obtained in this way).
\begin{theorem}\label{itn}
For $f\in\y$, the function
\begin{equation}\label{e1}
\rohat=\sum_{k=0}^\infty\frac1{(f^k)'(0)}\chi_{_{[f^k(0),1]}}
\end{equation}
is the density of an acim for $f$.
\end{theorem}
By~\cite{K}, $\rohat$ is bounded below by a positive constant on the
support of the acim, which is the union of finitely many intervals.
For $f\in\x$, this support is just $[0,1]$.

We have to explain how to understand the derivatives in~\eqref{e1} if
0 is periodic. In that case we take the limit of $(f^k)'(x)$ as $x$
goes to 0 from the right. If the period of $0$ is $p$, then we look at
the map $f^{p-2}$ in a small neighborhood of $0$ (observe that
$f^{p-2}(0)=c$). If this map preserves orientation, then we should
interpret $f'(c)$ as $-t$; if it reverses orientation, then we should
interpret $f'(c)$ as $s$.

Similarly, we define the kneading sequence as the itinerary of $1$,
using symbols $L,R$ (but not $C$), but if $1$ is periodic, then we
take the limit of the itineraries of $x$ as $x$ goes to 1 from the
left. This is consistent with our treatment of the derivative at $c$,
and this is how the kneading sequence is treated in~\cite{ITN} (except
that the authors use there $0,1$ instead of $L,R$).

Let us estimate the variation of $\rohat$ for $f\in\x$. Since $f(0)$
is to the left of the fixed point of $f$, there exists $m\ge 0$ such
that the kneading sequence of $f$ is $RLR^mL\dots$. We will denote the
space of all elements of $\x$ with the kneading sequence
$RLR^mL\dots$, where $m\le n$, by $\x_n$. Then $\x$ is the union of
the ascending sequence of subsets $\x_n$.

\begin{lemma}\label{varrohat}
If $f\in\x_n$ and $s\ge t$ then
\begin{equation}\label{e4}
\var(\rohat)\le\frac{n+1}s+\frac{2s+1}{s^2}.
\end{equation}
\end{lemma}

\begin{proof}
Clearly, $\var(\rohat)$ is equal to the sum of expressions
$1/|(f^k)'(0)|$ over those $k$ for which $f^k(x)\in(0,1)$. Therefore,
\begin{equation}\label{e2}
\var(\rohat)\le\sum_{k=1}^\infty\frac1{|(f^k)'(0)|}.
\end{equation}
This series converges because the map $f$ is piecewise expanding.

If $f\in\x_n$, then we have $|(f^k)'(0)|\ge s$ for $k=0,1,\dots,n$ and
$|(f^k)'(0)|\ge s^2t^{k-n-1}$ for $k>n$, so
\begin{equation}\label{e3}
\var(\rohat)\le\frac{n+1}s+\frac1{s^2}\left(1+\frac1t+\frac1{t^2}+
\dots\right)=\frac{n+1}s+\frac1{s^2}\cdot\frac{t}{t-1}.
\end{equation}

It is easy to calculate that $f(0)$ being to the left of the fixed
point of $f$ is equivalent to $st^2-s-t\ge 0$. Solving this inequality
for $t$ we get
\[
t>\frac{1+\sqrt{1+4s^2}}{2s}\ge 1+\frac1{2s}.
\]
Since the function $t\mapsto\frac{t}{t-1}$ is decreasing, we get
\[
\frac{t}{t-1}<\frac{1+\frac1{2s}}{\frac1{2s}}=2s+1.
\]
Therefore, for $f\in\x_n$ we get the estimate~\eqref{e4}.
\end{proof}

Now we normalize $\rohat$, that is, we set $\rho=\rohat/\int\rohat$.
Then $\rho$ is the density of the acip for $f$. Remember that
by~\cite{LiY}, the acip is unique.

\begin{theorem}\label{varrho}
For every $n$ and $\eps>0$ there exists $s(\eps,n)$ such that if
$f\in\x_n$ and $s\ge s(\eps,n)$ then $1-\eps\le\rho\le 1+\eps$ and
$\var(\rho)\le\eps$.
\end{theorem}
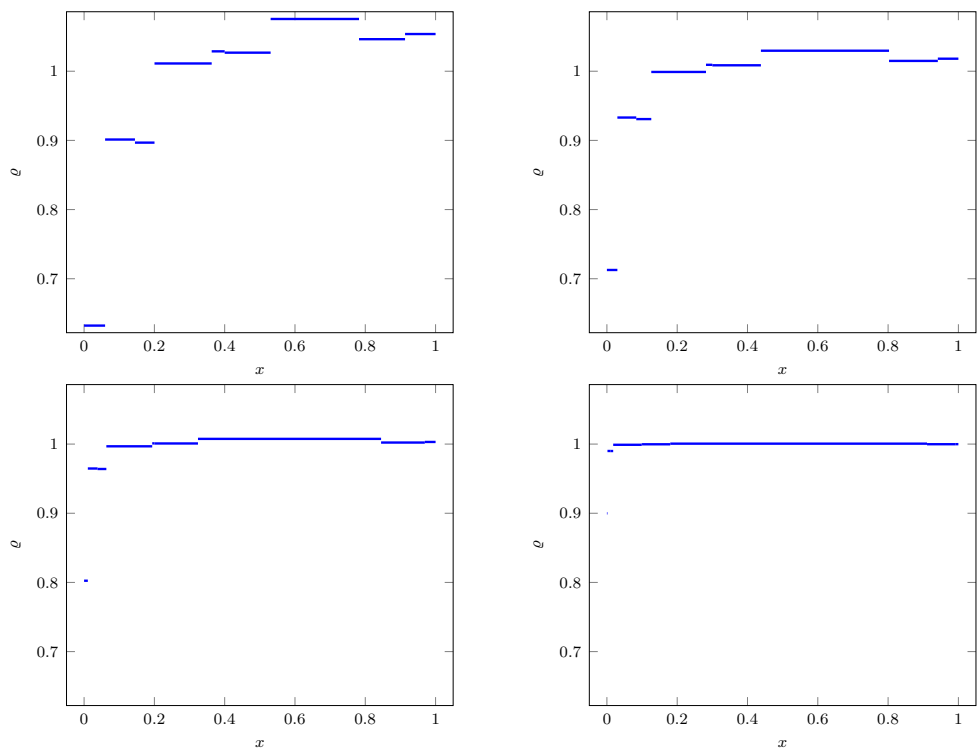
\begin{figure}[t]
\begin{tikzpicture}[scale=0.75, font={\tiny}]
\begin{axis}[xlabel={$x$}, xmin=-0.05, xmax=1.05,
             ylabel={$\rho$}, ymin=0.622, ymax=1.086]
\draw[blue, very thick] (0,0.632258) -- (0.059794,0.632258);
\draw[blue, very thick] (0.059794,0.901245) -- (0.144737,0.901245);
\draw[blue, very thick] (0.144737,0.896771) -- (0.200341,0.896771);
\draw[blue, very thick] (0.200341,1.011208) -- (0.363054,1.011208);
\draw[blue, very thick] (0.363054,1.028735) -- (0.400000,1.028735);
\draw[blue, very thick] (0.400000,1.026832) -- (0.530699,1.026832);
\draw[blue, very thick] (0.530699,1.075518) -- (0.782168,1.075518);
\draw[blue, very thick] (0.782168,1.046306) -- (0.913158,1.046306);
\draw[blue, very thick] (0.913158,1.053763) -- (1,1.053763);
\end{axis}
\end{tikzpicture}\hspace*{2em}
\begin{tikzpicture}[scale=0.75, font={\tiny}]
\begin{axis}[xlabel={$x$}, xmin=-0.05, xmax=1.05,
             ylabel={$\rho$}, ymin=0.622, ymax=1.086]
\draw[blue, very thick] (0,0.712664) -- (0.029821,0.712664);
\draw[blue, very thick] (0.029821,0.933035) -- (0.083545,0.933035);
\draw[blue, very thick] (0.083545,0.930800) -- (0.126261,0.930800);
\draw[blue, very thick] (0.126261,0.998943) -- (0.281916,0.998943);
\draw[blue, very thick] (0.281916,1.009268) -- (0.300000,1.009268);
\draw[blue, very thick] (0.300000,1.008577) -- (0.438139,1.008577);
\draw[blue, very thick] (0.438139,1.029649) -- (0.802659,1.029649);
\draw[blue, very thick] (0.802659,1.014899) -- (0.941518,1.014899);
\draw[blue, very thick] (0.941518,1.018091) -- (1,1.018091);
\end{axis}
\end{tikzpicture}

\begin{tikzpicture}[scale=0.75, font={\tiny}]
\begin{axis}[xlabel={$x$}, xmin=-0.05, xmax=1.05,
             ylabel={$\rho$}, ymin=0.622, ymax=1.086]
\draw[blue, very thick] (0,0.802458) -- (0.010653,0.802458);
\draw[blue, very thick] (0.010653,0.964678) -- (0.038277,0.964678);
\draw[blue, very thick] (0.038277,0.963992) -- (0.063352,0.963992);
\draw[blue, very thick] (0.063352,0.996785) -- (0.193810,0.996785);
\draw[blue, very thick] (0.193810,1.001028) -- (0.200000,1.001028);
\draw[blue, very thick] (0.200000,1.000889) -- (0.324038,1.000889);
\draw[blue, very thick] (0.324038,1.007518) -- (0.844952,1.007518);
\draw[blue, very thick] (0.844952,1.002215) -- (0.969378,1.002215);
\draw[blue, very thick] (0.969378,1.003072) -- (1,1.003072);
\end{axis}
\end{tikzpicture}\hspace*{2em}
\begin{tikzpicture}[scale=0.75, font={\tiny}]
\begin{axis}[xlabel={$x$}, xmin=-0.05, xmax=1.05,
             ylabel={$\rho$}, ymin=0.622, ymax=1.086]
\draw[blue, very thick] (0,0.899887) -- (0.001629,0.899887);
\draw[blue, very thick] (0.001629,0.990022) -- (0.009853,0.990022);
\draw[blue, very thick] (0.009853,0.989956) -- (0.017894,0.989956);
\draw[blue, very thick] (0.017894,0.998984) -- (0.099112,0.998984);
\draw[blue, very thick] (0.099112,0.999717) -- (0.100000,0.999717);
\draw[blue, very thick] (0.100000,0.999710) -- (0.180281,0.999710);
\draw[blue, very thick] (0.180281,1.000615) -- (0.910799,1.000615);
\draw[blue, very thick] (0.910799,0.999801) -- (0.991132,0.999801);
\draw[blue, very thick] (0.991132,0.999874) -- (1,0.999874);
\end{axis}
\end{tikzpicture}
\caption{The graphs of the density $\rho$ for:\newline
\hspace*{2em}top left: $s = 2.35051\ldots$ and $t = \tfrac{10}{6};$\newline
\hspace*{2em}top right: $s = 3.2339\ldots$ and $t = \tfrac{10}{7};$\newline
\hspace*{2em}bottom left: $s = 4.9467\ldots$ and $t = \tfrac{10}{8};$\newline
\hspace*{2em}bottom right: $s = 9.9837\ldots$ and $t = \tfrac{10}{9}$.}\label{Densities}
\end{figure}

To illustrate Theorem~\ref{varrho} we show in Figure~\ref{Densities} four
densities for maps from $\x_0$ with larger and larger slope $s$.

\begin{proof}[Proof of Theorem~\ref{varrho}]
Let us start by estimating the value of $\rohat(0)$. If 0 is periodic
of period $p$, then
\begin{equation}\label{e5}
|\rohat(0)-1|\le\sum_{i=1}^\infty\frac1{|(f^{ip})'(0)|}\le
\sum_{i=1}^\infty\frac1{s^i}=\frac1{s-1}.
\end{equation}
If 0 is not periodic, then $\rohat(0)=1$, so~\eqref{e5} also holds.

Now, the existence of the required $s(\eps,n)$ follows from the fact
that the limit as $s\to\infty$ of the right hand side of
both~\eqref{e4} and~\eqref{e5} is 0.
\end{proof}

We want to justify the assumption that $f\in \x_n$ in Theorem~\ref{varrho}.
For this we show that if we move through different classes $\x_n$ then the
limit density can be completely different (see Figure~\ref{hiperbolafig}).

\begin{figure}[hb]
\hfill
\begin{tikzpicture}[scale=0.9]
\begin{axis}[xlabel={$x$}, xmin=0, xmax=1,
             ylabel={$\varrho$}, ymin=0.77, ymax=1.33,
             enlargelimits=0.025]
\addplot[blue, very thick, domain=0:0.4, samples=1000] {0.6570787563 + 0.0657078756/(0.5-x)};
\addplot[blue, very thick, domain=0.6:1, samples=1000] {0.6570787563 + 0.0657078756/(x-0.5)};
\draw[blue, very thick] (0.3975, 1.3141575126) -- (0.6025, 1.3141575126);
\end{axis}
\end{tikzpicture}
\hfill
\begin{tikzpicture}[scale=0.9]
\begin{axis}[xlabel={$x$}, xmin=0, xmax=1,
             ylabel={$\varrho$}, ymin=0.9, ymax=1.85,
             enlargelimits=0.025]
\addplot[blue, very thick, domain=0:0.49, samples=1000] {0.9105474040 + 0.0091054740/(0.5-x)};
\addplot[blue, very thick, domain=0.51:1, samples=1000] {0.9105474040 + 0.0091054740/(x-0.5)};
\draw[blue, very thick] (0.48725, 1.8210948081) -- (0.5125, 1.8210948081);
\end{axis}
\end{tikzpicture}
\hfill\strut
\caption{Two limit densities for
$f(0) = a = 0.4$ (left picture) and
$f(0) = a = 0.49$ (right picture).}\label{hiperbolafig}
\end{figure}
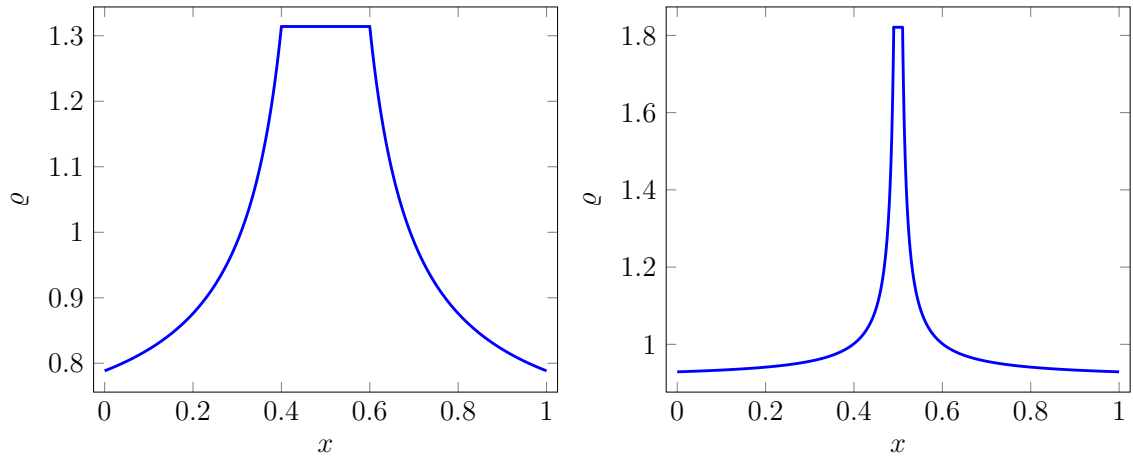

Although we agreed to work with the version of the kneading theory
that does not use the symbol $C$, when working with maps for which 0
is periodic, it is sometimes convenient to use $C$. Let us fix
$a\in(0,1/2)$ and take the skew tent map $f$ with $f(0)=a$ and the
kneading sequence $RLR^{2n}C$ (then the orbit of the turning point is
the \v{S}tefan periodic orbit of period $2n+3$). For $k=0,1,\dots,2n$,
set
\begin{equation}\label{h1}
x_k=f^k(a)=\frac{t}{t+1}+\left(a-\frac{t}{t+1}\right)(-t)^k.
\end{equation}
Let $\roov$ be the density of the acim, normalized by $\roov(0)=1$
(so, since the only preimage of 0 is 1, and $\roov$ is invariant
under the Frobenius-Perron operator, we have $\roov(1)=t$). The
function $\roov$ is constant on $[a,f(a)]$ and has jumps
\[
\frac1{s(-t)^k}=\frac{t-1}{(1-a)(-1)^kt^{k+1}}
\]
at $x_k$, since
\[
s=\frac{1-a}c=\frac{1-a}{1-\frac1t}=\frac{t(1-a)}{t-1}.
\]

Let us concentrate on the right part of the interval (the situation on
the left one is very similar). If $\pov$ is the value of $\roov$ on
$[a,f(a)]$, then the value of $\roov$ at $x_{2k+1}$
(more precisely, the limit from the left) is
\[\begin{split}
\roov(x_{2k+1})&=\pov-\sum_{j=0}^{k-1}\frac{t-1}{(1-a)t^{2j+2}}\\
&=\pov-\frac{t-1}{(1-a)t^2}\cdot\frac{1-1/t^{2k}}{1-1/t^2}=
\pov-\frac1{(1-a)(t+1)}\left(1-\frac1{t^{2k}}\right).
\end{split}\]
By~\eqref{h1},
\[
\frac1{t^{2k}}=\frac{t\left(\frac{t}{t+1}-a\right)}
{x_{2k+1}-\frac{t}{t+1}},
\]
so
\[
\roov(x_{2k+1})=\pov-\frac1{(1-a)(t+1)}\left(1-\frac{t
\left(\frac{t}{t+1}-a\right)}{x_{2k+1}-\frac{t}{t+1}}\right).
\]

Since $x_{2n+1}=1$, we have
\[
\pov=t+\frac1{(1-a)(t+1)}\left(1-\frac{t
\left(\frac{t}{t+1}-a\right)}{1-\frac{t}{t+1}}\right)=\frac1{1-a}.
\]

Now let us go with $n$ to infinity, keeping $a$ constant. Then $t$
goes to 1, $f(a)$ goes to $1-a$, and for every $\eps>0$ the points
$x_{2k+1}$ are $\eps$-dense in $[1-a,1]$ if $n$ is large enough.
Therefore, $\roov$ goes to the limit $\rotil$, and for
$x\ge 1-a$ we have
\[
\rotil(x)=\frac1{1-a}-\frac1{2(1-a)}\left(1-\frac{\frac12-a}
{x-\frac12}\right)=\frac1{2(1-a)}+\frac{1-2a}{4(1-a)}\cdot
\frac1{x-\frac12}.
\]
Very similar computations give us
\[
\rotil(x)=\frac1{2(1-a)}+\frac{1-2a}{4(1-a)}\cdot
\frac1{\frac12-x}
\]
for $x\le a$.

In such a way we get
\begin{equation}\label{h2}
\rotil(x)=\begin{cases}
\frac1{2(1-a)}+\frac{1-2a}{4(1-a)}\cdot\frac1{\frac12-x}&
\text{if $x<a$,}\\
\frac1{1-a}& \text{if $a\le x\le 1-a$,}\\
\frac1{2(1-a)}+\frac{1-2a}{4(1-a)}\cdot\frac1{x-\frac12}&
\text{if $x>1-a$,}
\end{cases}\end{equation}
so in particular, $\rotil$ is symmetric with respect to $1/2$.

Now we have to normalize $\rotil$ (that is, to divide it by its
integral) in order to get the limit density $\rho$. We have
\[\begin{split}
\int\rotil(x)\;dx&=2\int_0^a\left(\frac1{2(1-a)}+\frac{1-2a}
{4(1-a)}\cdot\frac1{\frac12-x}\right)\;dx+\frac{1-2a}{1-a}\\
&=\frac{a}{1-a}+\frac{1-2a}{1-a}+\frac{1-2a}{2(1-a)}
\int_0^a\frac1{\frac12-x}\;dx\\
&=1-\frac{1-2a}{2-2a}\log\left(\frac12-x\right)\bigg|_0^a
=1-\frac{1-2a}{2-2a}\log(1-2a).
\end{split}\]

Thus the limit density $\rho$ of the acip is given by the function
from formula~\eqref{h2} divided by $1-(1-2a)\log(1-2a)/(2-2a)$.

Observe that as $a$ goes to $1/2$ then the maximal value of $\rho$ (at
the plateau) increases to 2. This motivates us to make the following
conjecture, with weaker and stronger versions.

\begin{conjecture}\label{conj-w}
There exists a constant $K$ such that for every $f\in\x$ the density
of the acip for $f$ is bounded above by $K$.
\end{conjecture}

\begin{conjecture}\label{conj-s}
For every $f\in\x$ the density of the acip for $f$ is bounded above by
$2$.
\end{conjecture}

\section{Continuous dependence of densities}\label{s-dens}

The next thing to show is that the density $\rho$ of the acip depends
continuously on the map. In this section we will stress the dependence
on the map, so we will write $\rho_f$, $\rohat_f$, etc.

For $\rohat_f=\rohat$ given by~\eqref{e1}, let us look at its
approximations $\rohat_{\ell,f}$, given by
\begin{equation}\label{e1a}
\rohat_{\ell,f}=\sum_{k=0}^\ell\frac1{(f^k)'(0)}\chi_{_{[f^k(0),1]}}.
\end{equation}

We will be proving that $\rohat_g$ as a function of $g$ is continuous
at any given $f$. Thus we fix $f\in\y$. We will denote the ball of
radius $\delta$ in $\y$ (in the sup metric), centered at $f$ by
$B(f,\delta)$. We will also denote the $L^1$ norm by $\|\cdot\|$.

\begin{theorem}\label{cont-dens}
The map $g\mapsto\rho_g$ from $\y$ to $L^1([0,1])$ is continuous.
\end{theorem}

\begin{proof}
Fix $f\in\y$ and $\eps>0$. There exists constants $T>1$ and $\delta>0$
such that if $g\in B(f,\delta)$ then the absolute values of the slopes
of $g$ are at least $T$. Then the sup distance between $\rohat_g$ and
$\rohat_{\ell,g}$ is not larger than
\[
\sum_{k=\ell+1}^\infty\frac1{T^k}=\frac1{T^\ell(T-1)}.
\]
Let us choose $\ell$ so large that this is less than $\eps/3$.

When we want to estimate the distance between $\rohat_{\ell,f}$ and
$\rohat_{\ell,g}$ (with $g$ close to $f$), the uniform distance will
not work well. This is due to the fact that it may happen that
$g(0)\ne f(0)$, so the distance between $\chi_{_{[g(0),1]}}$ and
$\chi_{_{[f(0),1]}}$ is 1. This is the reason why we are using the
$L^1$ distance. Note that our interval has length 1, so the $L^1$
distance is not larger than the uniform distance.

Making $\delta$ smaller if
necessary, we can guarantee that if $g\in B(f,\delta)$ then
\[
\sum_{k=1}^\ell|f^k(0)-g^k(0)|<\eps/3,
\]
so then $\|\rohat_{\ell,f}-\rohat_{\ell,g}\|<\eps/3$. This proves that
for a sufficiently small $\delta$, if $g\in B(f,\delta)$, then
$\|\rohat_f-\rohat_g\|<\eps$. Hence, the map $g\mapsto\rohat_g$ is
continuous at $f$.

Now we have to switch from $\rohat$ to
$\rho=\rohat/\int\rohat=\rohat/\|\rohat\|$. We have
\[\begin{split}
&\|\rho_f-\rho_g\|=\left\|\frac{\rohat_f}{\|\rohat_f\|}-
\frac{\rohat_g}{\|\rohat_g\|}\right\|\le\left\|\frac{\rohat_f}
{\|\rohat_f\|}-\frac{\rohat_g}{\|\rohat_f\|}\right\|+
\left\|\frac{\rohat_g}{\|\rohat_f\|}-\frac{\rohat_g}
{\|\rohat_g\|}\right\|\\
=\ &\frac{\|\rohat_f-\rohat_g\|}{\|\rohat_f\|}+\|\rohat_g\|
\frac{\big|\|\rohat_g\|-\|\rohat_f\|\big|}{\|\rohat_f\|
\cdot\|\rohat_g\|}=\frac{\|\rohat_f-\rohat_g\|+
\big|\|\rohat_g\|-\|\rohat_f\|\big|}{\|\rohat_f\|}
\le\frac{2\|\rohat_f-\rohat_g\|}{\|\rohat_f\|}.
\end{split}\]
Thus, the map $\rohat_g\mapsto\rho_g$ is continuous at $\rohat_f$, so
it is continuous everywhere (that is, on the image of the map
$g\mapsto\rohat_g$).

This proves that the map $g\mapsto\rho_g$ is continuous on $\y$.
\end{proof}

\section{Continuous dependence of metric entropy}\label{s-entr}

For $g\in\y$ denote by $\mu_g$ the acip for $g$.

\begin{theorem}\label{cont-metric-ent}
The map $g\mapsto h_{\mu_g}(g)$ from $\y$ to $\R$ is continuous.
\end{theorem}

\begin{proof}
By the Rohlin Lemma (\cite{P, R}), we have
\[
  h_{\mu_f}(f) = \int\log\abs{f'}\; d\mu_f
           = \int\log\abs{f'}\rho_f\; d\lambda,
\]
where $\lambda$ is the Lebesgue measure.

Fix $f\in\y$. The density $\rho_f$ is bounded above by some constant
$M$. There are also constants $N,\eta$ such that for every $g\in
B(f,\eta)$ the function $\log|g'|$ is bounded above by $N$.

On the other hand, in view of Theorem~\ref{cont-dens} and since the
map $g\mapsto\log\abs{g'}$ from $\y$ to $L^1([0,1])$ is continuous,
for every $\eps > 0$ there exists $\delta\in(0,\eta)$ such that
\begin{align*}
\big\|\log\abs{f'} - \log\abs{g'}\big\| & < \frac{\eps}{2M},\text{
  and}\\
\|\rho_f-\rho_g\| & < \frac{\eps}{2N},
\end{align*}
whenever $g \in B(f,\delta)$. Thus,
\begin{align*}
\abs{h_{\mu_f}(f) - h_{\mu_g}(g)} & = \abs{\int\log\abs{f'}\rho_f\;
d\lambda - \int\log\abs{g'}\rho_g\; d\lambda}\\
& \leq \int\big|\log\abs{f'} - \log\abs{g'}\big|\;\rho_f\; d\lambda +
\int\log\abs{g'}\abs{\rho_f - \rho_g}\; d\lambda\\
& \le M \big\|\log\abs{f'} - \log\abs{g'}\big\| + N\|\rho_f - \rho_g\|
< \eps.
\end{align*}
\end{proof}

\section{Rectangular root}\label{s-root}

To switch from a map $g$ to a map with the topological and metric
entropies halved, we need a square root procedure. The traditional one
works as follows. Remember that when we say ``linear'' we mean
what in algebra is called ``affine''.

Suppose we have a unimodal map $g$ with an acip $\nu$. Then we define
a unimodal map $G$ in the following way. On $[0,1/3]$,
$G(x)=(2+g(1-3x))/3$, on $[2/3,1]$, $G(x)=1-x$, on $[1/3,2/3]$ $G$ is
linear to make it continuous. Figure~\ref{fig1} shows $g$, $G$ and
$G^2$. For $G^2$ the intervals $[0,1/3]$ and $[2/3,1]$ are invariant,
and on each of them $G^2$ is linearly conjugate to $g$. The map $G$
maps each of them to the other one. All points of $[1/3,2/3]$, except
the fixed point, are eventually mapped to $[0,1/3]\cup[2/3,1]$ (except
in the case when all those points are fixed points of $G^2$). Therefore
$h(G)=(1/2)h(g)$ and $h_\kappa(G)=(1/2)h_\nu(g)$, where $\kappa$ is the
acip for $G$ obtained from $\nu$.
\begin{figure}[h]
\begin{center}
\includegraphics[height=48truemm]{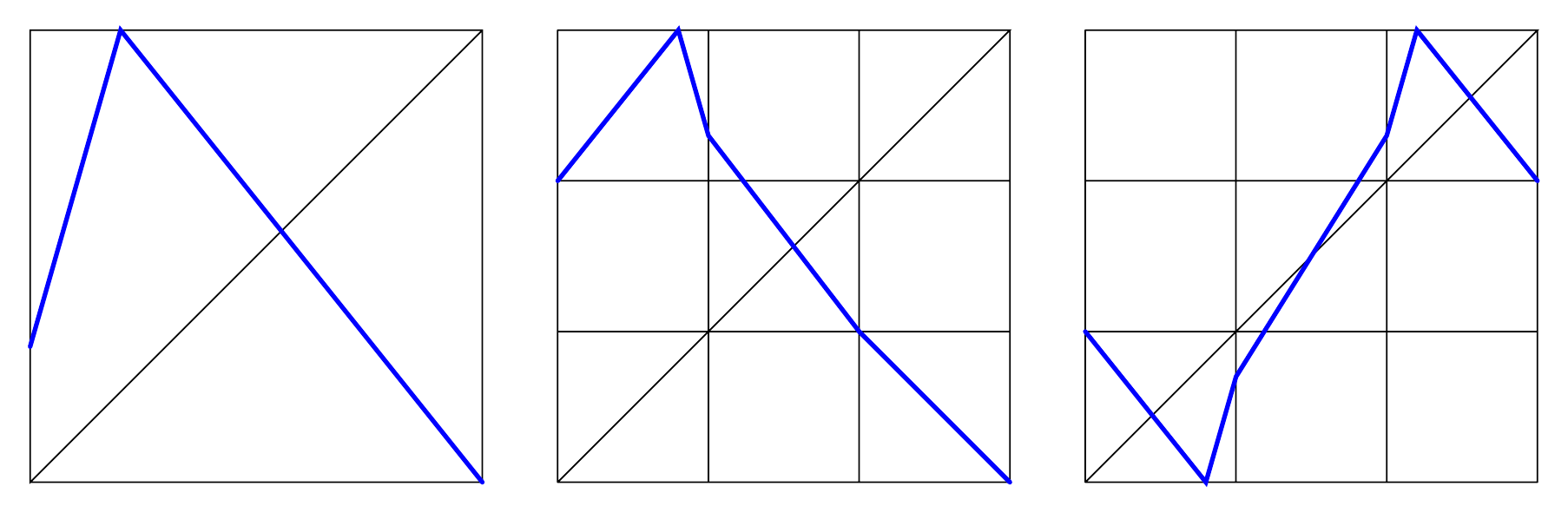}
\caption{The square root procedure. Maps $g$, $G$ and $G^2$.}\label{fig1}
\end{center}
\end{figure}

The problem with this procedure is that even if we start with a
piecewise expanding map, we end up with a map that has intervals on
which the absolute value of the slope is 1. To modify this procedure
in such a way that the resulting map is also piecewise expanding, let
us look at the original procedure from the point of view of
compositions of maps.

Let $\phi$ be the linear, orientation reversing map, that sends the
interval $[0,1/3]$ onto the interval $[0,1]$. Then $\phi(x)=1-3x$ and
$\phi^{-1}(x)=(1-x)/3$. Similarly, let $\psi$ be the linear
orientation preserving map, sending the interval $[2/3,1]$ onto the
interval $[0,1]$. Then $\psi(x)=3x-2$ and $\psi^{-1}(x)=(2+x)/3$. Now
we see that on $[0,1/3]$ we have $G=\psi^{-1}\circ g\circ\phi$ and on
$[2/3,1]$ we have $G=\phi^{-1}\circ\psi$. Since $G$ sends $[0,1/3]$
onto $[2/3,1]$ and vice versa, on $[0,1/3]$ we have
$G^2=\phi^{-1}\circ g\circ\phi$ and on $[2/3,1]$ we have
$G^2=\psi^{-1}\circ g\circ\psi$. This explains why on both intervals
$G^2$ is linearly conjugate to $g$.

If you look at Figure~\ref{fig1}, you see nine small squares in a big
square. If we change some of them into rectangles, our procedure will
work better. Therefore we will call this procedure the
\emph{rectangular root}.

If $g(0)>0$ then we define the rectangular root procedure by changing
in the definition of $\phi$ the interval $[0,1/3]$ to $[0,(1+\eps)/3]$
for some sufficiently small $\eps>0$ (see Figure~\ref{fig2}). The
slopes of $G$ on $[0,(1+\eps)/3]$ are now equal to the slopes of $g$
divided by $1+\eps$, so if $\eps$ is sufficiently small, their
absolute values are still larger than 1. The slope of $G$ on $[2/3,1]$ is
$-(1+\eps)$. The interval $[(1+\eps)/3,2/3]$ is mapped by $G$ onto a
larger interval, so the absolute value of the slope is larger than 1
as well.
\begin{figure}[h]
\begin{center}
\includegraphics[height=48truemm]{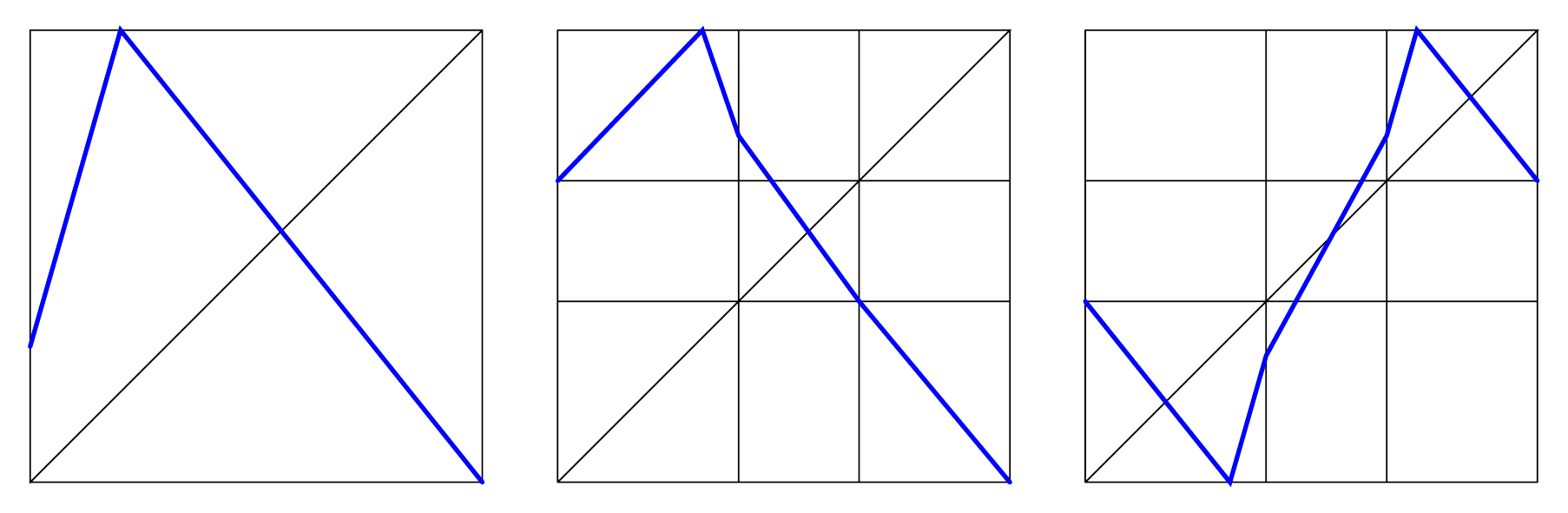}
\caption{The rectangular root procedure when $g(0)>0$. Maps $g$, $G$
  and $G^2$.}\label{fig2}
\end{center}
\end{figure}

If $g(0)=0$, then we simply remove the middle interval. That is, we
take as $\phi$ the linear orientation reversing maps, sending the
interval $[0,(1+\eps)/2]$ onto the interval $[0,1]$, and as $\psi$ the
linear orientation preserving map, sending the interval
$[(1+\eps)/2,1]$ onto the interval $[0,1]$. Then we set
$G=\psi^{-1}\circ g\circ\phi$ on $[0,(1+\eps)/2]$ and
$G=\phi^{-1}\circ\psi$ on $[(1+\eps)/2,1]$ (see Figure~\ref{fig3}).
As in the first case, if $\eps>0$ is sufficiently
small, then $G$ is piecewise expanding.
\begin{figure}[h]
\begin{center}
\includegraphics[height=48truemm]{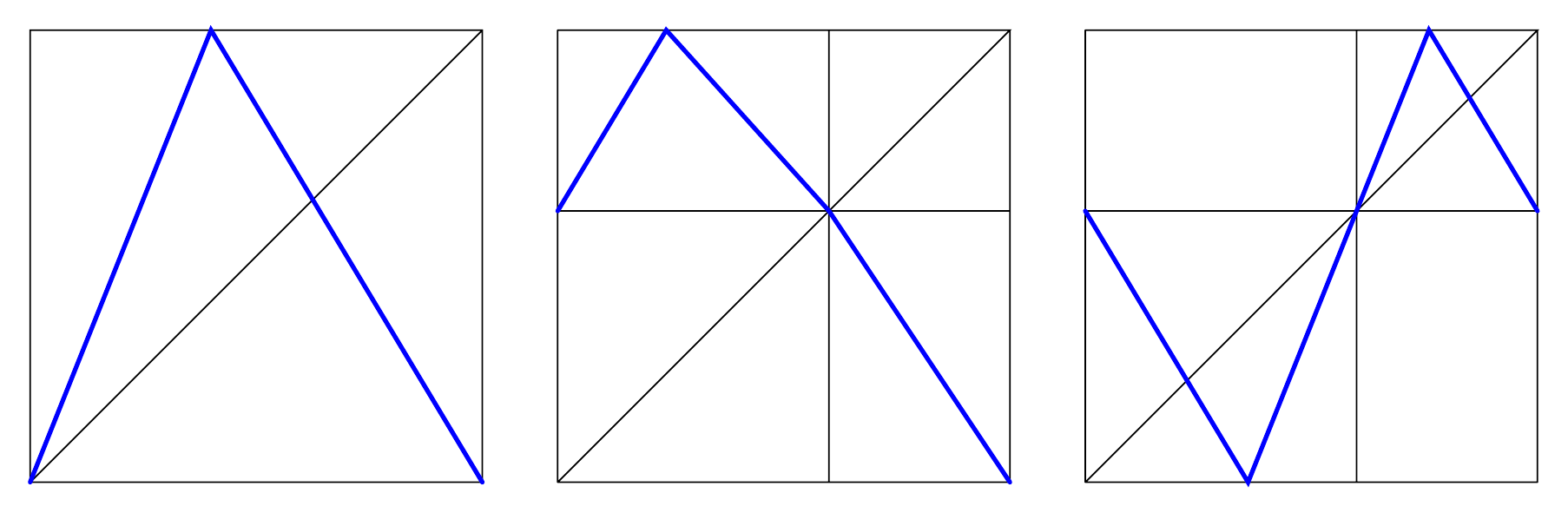}
\caption{The rectangular root procedure when $g(0)=0$. Maps $g$, $G$
  and $G^2$.}\label{fig3}
\end{center}
\end{figure}

\section{Main Theorems}\label{s-main}

The first flexibility result is about mixing skew tent maps (the
space $\x$).

\begin{MainTheorem}\label{flex-skew}
For every pair $a,b\in\R$ with $\frac12\log2<a\leq\log2$ and $0<b\leq
a$ there exists a piecewise expanding mixing skew tent map $f$ for
which $h(f)=a$ and $h_\mu(f)=b$, where $\mu$ is the acip for $f$.
\end{MainTheorem}

\begin{proof}
If $\log s_a=a$, then $f\in\x$ with slopes $s_a$ and $-s_a$ has
topological entropy $a$. By~\cite[Theorem~C]{MV} and the subsequent
remark, there exists a number $\gamma>1$ and a continuous decreasing
function $\beta:(1,\gamma]\to[1,\infty)$ such that
$\lim_{t\searrow 1} \beta(t) = \infty$ and for $g\in\x$ we have
$h(g)=a$ if and only if the slopes of $g$ are $\beta(t),-t$ for
$t \in (1,\gamma]$. This in particular implies that $s_a\le\gamma$ and
$\beta(s_a)=s_a$. Moreover, by~\cite[Theorem~C]{MV}, the skew tent
maps with the slopes $\beta(t),-t$ have the same kneading sequence for
all $t$, so we will be able to use Theorem~\ref{varrho} for them.

Let $f_t\in\x$ be the function with slopes $\beta(t)$ and $-t$, and
let $\mu_t$ be its acip. For $t=s_a$ this measure is also the measure
with maximal entropy, so $h_{\mu_t}(f_t)=a$. As $t$ goes to 1, then
$\beta(t)$ goes to infinity and the turning point $c_t$ goes to 0, so
by Theorem~\ref{varrho} $\mu_t([0,c_t])$ goes to 0. Since the
partition of $[0,1]$ into $[0,c_t]$ and $(c_t,1]$ is a generator, this
implies that $h_{\mu_t}(f_t)$ goes to 0. Therefore, by
Theorem~\ref{cont-metric-ent} and continuity of the function $\beta$,
there exists $t$ such that $h_{\mu_t}(f_t)=b$.
\end{proof}

The next theorem shows flexibility of entropies for piecewise
expanding unimodal maps.

\begin{MainTheorem}\label{flex-uni}
For every pair $a,b\in\R$ with $0<a\leq\log2$ and $0<b\leq a$ there
exists a piecewise expanding unimodal map $f$ for which $h(f)=a$ and
$h_\mu(f)=b$, where $\mu$ is the acip for $f$.
\end{MainTheorem}

\begin{proof}
If $a>\frac12\log 2$, this follows from Theorem~\ref{flex-skew}. If
$a\le\frac12\log 2$, there is $n\ge 1$ such that
$\frac12\log 2<2^n a\le\log2$. Then use Theorem~\ref{flex-skew} to
find $g\in\x$ with $h(g)=2^n a$ and $h_\nu(g)=2^n b$, where $\nu$ is
the acip for $g$. Finally, use $n$ times the rectangular root
procedure to get the desired $f$. Each time we use this procedure,
both topological and metric entropy get divided by 2. Then $f$
is a piecewise expanding unimodal map with $h(f)=a$ and $h_\mu(f)=b$.
\end{proof}

The following proposition shows that in Theorem~\ref{flex-uni} we
cannot replace unimodal maps by skew tent maps.

\begin{proposition}\label{example}
There is no piecewise expanding skew tent map with topological entropy
$\frac14\log 2$ and metric entropy smaller than
$\frac14\log\frac{16}{15}$.
\end{proposition}

\begin{proof}
Let $f\in\y$ have entropy $h(f)=\frac14\log 2$. This means that $f$ is
twice renormalizable (see, e.g.,~\cite{CE}). If the slopes of $f$ are
$s$ and $-t$, after the first renormalization we get a skew tent map
with slopes $t^2$ and $-st$ and entropy $\frac12\log 2$, linearly
conjugate to $f^2$ restricted to an invariant interval. After the
second renormalization we get a skew tent map $g$ with slopes $s^2t^2$
and $-st^3$ and entropy $\log 2$, linearly conjugate to $f^4$
restricted to an invariant interval.

Entropy $\log 2$ for a skew tent map means that the sum of reciprocals
of the absolute values of the slopes is 1, and that the acip is the
Lebesgue measure. Thus, we have $s+t=s^2t^3$. Remember that we assume
$s,t>1$.

Set $T=\max(s,t)$. We have $T^2<s^2t^3=s+t\le 2T$, so $T<2$.
Hence, absolute values of both slopes of the second renormalization
$g$ are smaller than 16.
Therefore $1/16<c<15/16$, so either $\lambda([0,c])>1/16$ or
$\lambda((c,1])>1/16$. Then,
\[
h_\lambda(g)>-\frac1{16}\log\frac1{16}-\frac{15}{16}\log\frac{15}{16}
=\log 16-\frac{15}{16}\log 15>\log 16-\log 15=\log\frac{16}{15}.
\]
Thus, $h(f)=\frac14 h(g)>\frac14\log\frac{16}{15}$.
\end{proof}

\section{An interesting observation}\label{s-obs}

Hidden in~\cite{ITN} is the formula called by the authors
``$f$-expansion''. Namely, for $f\in\y$ and $x\in[0,1]$ we have,
(translating to our notation)
\begin{equation}\label{e6a}
x=1-\frac1t\sum_{k=0}^\infty\frac1{(f^k)'(x)}.
\end{equation}
It is easy to see that this can be reinterpreted as
\begin{equation}\label{e6b}
f(x)=\sum_{k=0}^\infty\frac1{(f^k)'(x)}
\end{equation}
for all $x>c$. However, since~\eqref{e6b} does not work for $x<c$, it
does not look important. On the other hand, if we take in~\eqref{e6a}
$x=1$, we get
\begin{equation}\label{e6}
\sum_{k=0}^\infty\frac1{(f^k)'(1)}=0,
\end{equation}
which is much more interesting.

If $s=t$, then the left-hand side of~\eqref{e6} is the value of the
Milnor-Thurston kneading determinant at $1/s$. However, $s$ is the
exponential of the topological entropy of $f$. Thus,
formula~\eqref{e6} is a generalization and a different interpretation
of the Milnor-Thurston formula connecting the topological entropy and
the kneading determinant for unimodal maps.

\end{document}